\documentclass[12pt,a4paper]{amsart}
\usepackage{amsfonts,color}
\usepackage{amsthm}
\usepackage{amsmath}
\usepackage{amscd}
\usepackage[utf8]{inputenc}
\usepackage{t1enc}
\usepackage[mathscr]{eucal}
\usepackage{indentfirst}
\usepackage{graphicx}
\usepackage{graphics}
\usepackage{pict2e}
\usepackage{epic}
\usepackage{url}
\usepackage{epstopdf}
\usepackage{comment}
\usepackage{hyperref}
\usepackage{tikz}
\usepackage{amsfonts}
\usepackage{wrapfig}
\usepackage{floatflt }
\usepackage{ amssymb }
\usetikzlibrary{positioning}

\usepackage{caption}
\usepackage{subcaption}

\numberwithin{equation}{section}
\usepackage[margin=2.6cm]{geometry}

\theoremstyle{plain}
\newtheorem{Th}{Theorem}[section]
\newtheorem{Lemma}[Th]{Lemma}
\newtheorem{Cor}[Th]{Corollary}
\newtheorem{Prop}[Th]{Proposition}

\theoremstyle{definition}
\newtheorem{Def}[Th]{Definition}

\newtheorem{Question}[Th]{Question}
\newtheorem{Rem}[Th]{Remark}
\newtheorem{Remark}[Th]{Remark}
\newtheorem{?}[Th]{Problem}

\newcommand{\ZZ}{\mathbb{Z}}
\newcommand{\F}{\mathbb{F}}

\newcommand{\m}{\mathcal}
\newcommand{\lef}{\left\lfloor}
\newcommand{\rrr}{\right\rfloor}

\makeatletter
\newcommand{\mainsectionstyle}{%
  \renewcommand{\@secnumfont}{\bfseries}
  \renewcommand\section{\@startsection{section}{1}%
	\z@{.7\linespacing\@plus\linespacing}{.5\linespacing}%
	{\normalfont\large\scshape\centering\bfseries}}

}
\makeatother

\begin{document}
\mainsectionstyle
\title{Clique number of Xor products of Kneser graphs}
\author{
András Imolay 
\and 
Anett Kocsis
\and
Ádám Schweitzer}
\date{}
\thanks{The first and the second authors are supported by the ÚNKP-20-1 New National Excellence Program of the Ministry for Innovation and Technology from the source of the National Research, Development and Innovation Fund.}

 \subjclass[2020]{Primary: 05D05. Secondary: 05C69, 05C76}

 \keywords{extremal set theory, intersecting families, Xor product} 
\begin{abstract}
In this article we investigate a problem in graph theory, which has an equivalent reformulation in extremal set theory similar to the problems researched in \cite{article} by Gyula O.H. Katona, who proposed our problem as well. In the graph theoretic form we examine the clique number of the Xor product of two isomorphic $KG(N,k)$ Kneser graphs. Denote this number with $f(k,N)$.
We give lower and upper bounds on $f(k,N)$, and we solve the problem up to a constant deviation depending only on $k$, and find the exact value for $f(2,N)$ if $N$ is large enough. Also we compute that $f(k,k^2)$ is asymptotically equivalent to $k^2$.  
\end{abstract}

\maketitle

\section{Introduction}

\noindent In the abstract we briefly stated the problem in the graph theoretic form, here we introduce the definitions that are necessary to understand the statement. 

\begin{Def}[Kneser graphs]
A Kneser graph $KG(N,k)$ is a graph defined as follows.

\noindent The vertices correspond to the $k$ element subsets of a given $N$ element set, and two vertices are connected if and only if the corresponding subsets are disjoint.
\end{Def}

The following graph product was studied in \cite{Thomason1997GraphPA}.
\begin{Def}[Xor product]
Given two graphs, $G$ and $H$, let us denote by $V(G)$, $V(H)$ and $E(G)$, $E(H)$ their vertex and edge sets respectively. Define the Xor product, denoted by $G \cdot H$, as the graph with the vertex set $V=V(G) \times V(H)$ (where $V(G) \times V(H)$ denotes the Cartesian product of the vertex sets), and two vertices $(g,h)$ and $(g',h')$ are connected if and only if among the statements $gg' \in E(G)$ and $hh' \in E(H)$ exactly one occurs. 
\end{Def} 

\begin{Def}[Clique number]
Let $G=(V,E)$ be a simple graph. A subset $H\subset V$ is called a clique, if any two vertices in $H$ are connected by an edge in $G$.
The clique number of $G$ is simply the size of a maximal clique, denoted by $\omega(G)$.
\end{Def}

So now we have everything to state the main problem of this article.

\begin{?}
We want to find the value of $\omega(KG(N,k) \cdot KG(N,k))$. 

\noindent As the exact value for all pairs $(k,N)$ is beyond our reach, we focus on giving lower and upper bounds, determining the asymptotic behaviour of these functions, and calculating the exact value for a large family of pairs $(k,N)$. 
\end{?}

Similar questions have been studied about the independence number of different, more conventional (Cartesian, Tensor and Strong) products of the Kneser graph in \cite{BRESAR20191017}. The clique number in these graphs for any conventional graph product can be determined easily as discussed in Subsection \ref{kneser}. 
In \cite{xorproduct} the limits of both the independence and clique numbers are studied for Xor powers.

Our problem can be rephrased as an extremal combinatorial problem. This was our original problem proposed by Gyula O.H. Katona. The problem is the following.

Let $k$ and $N$ be fixed positive integers. Let $A$ and $B$ be disjoint sets with $|A|=|B|=N$. We will consider families of subsets on the base set $A \cup B$. For easier notation we introduce the following.

\begin{Def} \label{S_A}
Throughout the article for a set $S \subset A \cup B$ we denote $S\cup A$ as $S_A$ and $S\cup B$ as $S_B$. 
\end{Def}

\begin{Def}[Semiintersecting family]
We call a family of sets $\mathcal{S}$ on the base set $A \cup B$ semiintersecting, if the following conditions hold.
\begin{itemize}
    \item $|S_A|=|S_B|=k$ for all $S \in \mathcal{S}$
    \item $(S_A) \cap (T_A)=\emptyset$ if and only if $(S_B) \cap (T_B)\neq \emptyset$ for all $S \neq T \in \mathcal{S}$ 
\end{itemize}
In other words the second condition states that two sets from $\m{S}$ intersect exactly at one side, so $S \cap T \neq \emptyset$ and either $S \cap T \subset A$ or $S \cap T \subset B$.

\noindent Let $f(k,N):=\max\{\ |\m{S}| \ : \ \m{S} \text{ is semiintersecting with constants $k$ and $N$} \}$. 
\end{Def}

The next proposition states the equivalence of the two problems.

\begin{Prop}
$$f(k,N)=\omega(KG(N,k) \cdot KG(N,k))$$
\end{Prop}

\begin{proof}
The proof is actually really easy if we understand the definitions.

Let $G_1=(V_1, E_1)$ and $G_2=(V_2, E_2)$ be two $KG(N,k)$ Kneser graphs, and let $A$ and $B$ be disjoint $N$-element sets corresponding to $G_1$ and $G_2$ respectively. Let $u=(u_1, u_2)$ and $w=(w_1, w_2)$ be two vertices of $G_1 \cdot G_2$. Denote the corresponding $k$-element sets of $u_1$, $u_2$, $w_1$, $w_2$ by $U_1$, $U_2$, $W_1$, $W_2$ respectively ($U_1, W_1 \subset A$, $U_2, W_2 \subset B$). $u$ and $v$ are connected in the product if one of the following occurs.
\begin{itemize}
    \item $u_1w_1 \in E_1$ and $u_2w_2 \notin E_2$, which is equivalent with $U_1 \cap W_1 = \emptyset$ and $U_2 \cap W_2 \neq \emptyset$.
    \item $u_1w_1 \notin E_1$ and $u_2w_2 \in E_2$, which is equivalent with $U_1 \cap W_1 \neq \emptyset$ and $U_2 \cap W_2 = \emptyset$.
\end{itemize}

So if we assign to each vertex $(u_1, u_2)$ of $G_1 \cdot G_2$ the set $U_1 \cup U_2$ then a clique in the graph corresponds to a semiintersecting family on $A \cup B$, and vice versa, we assigned a clique to each semiintersecting family on $A \cup B$. The proof is finished.
\end{proof}

Therefore our main question can be stated as follows.

\begin{?}
What is the value of $f(k,N)$?
\end{?}

Similar questions have been studied in extremal combinatorics about families of sets. Most notable among these would be the Erdős-Ko-Rado theorem \cite{zbMATH03162924}, which determines the size of the maximal uniform intersecting family of sets (a family of sets such that each set has the same size, and has a non-empty intersection with all other sets in the family).
In \cite{zbMATH03049937} this question has been studied for two part set systems. Here the maximal cardinality of such an intersecting family of sets is studied, where the base set is divided into two parts, and each set in the family has fixed size intersections with both parts. Other similar questions have been studied about such 2-part families of sets in \cite{Gerbner}. 

The connection of our problem to intersecting families is tighter than the definitions suggest. The proofs of the two main theorems, Theorem~\ref{f(2,N)} and Theorem~\ref{thm}, are based on the fact that in a sense the structure of a $k$-uniform large family with a lot of intersections is similar to the structure of an intersecting family, and an intersecting family cannot be too complicated.

The structure of intersecting families is widely studied in extremal combinatorics. One of the most famous results is due to Hilton and Milner \cite{hilton1967some} who gave an upper bound for the size of a non-trivially intersecting family. In \cite{balogh2015intersecting} the typical structure of an intersecting family is studied, and proved that assuming some conditions, almost all intersecting $k$-uniform families on $n$ element are trivial. This result has been improved since in several articles, the best known result is \cite{balogh2021intersecting} (preprint), where the authors only need the condition $n \geq 2k + 100 \ \text{log}(k)$. One can read more about the structure of intersecting families for example in this \cite{katona2020results} survey.

In \cite{gerbner2012almost} the maximal cardinality of almost intersecting families are examined, that is families with every set intersecting all other except at most $l$. This problem is closely related to the idea behind the proof of our main theorem.

In our article we only examine the case when the two Kneser graphs are isomorphic. 
The results of our article can be generalized to products of different Kneser graphs, though the proofs may be more technical.


\subsection{Structure of the article}

In the next subsections we talk about notations and preliminaries. We introduce the notations we use throughout the article and we state some definitions and theorems we will need in later chapters. Then, we study the maximal cliques in conventional products of graphs.

In chapter~\ref{upper} we talk about upper bounds. First we give two fairly easy bounds that work in every case. Then we solve the problem with up to constant deviation for large $N$, and give a sharp bound for the case of $k=2$ and large $N$.

In chapter~\ref{lower} we work on lower bounds. This mainly consists of giving as good constructions as we can. We find connections with other topics in combinatorics and algebra and use them to give good (sharp in a lot of cases) constructions, and we give a better lower bound for large $N$ than the trivial.

In chapter~\ref{summary} we compare the bounds from chapter~\ref{upper} and~\ref{lower} and summarize our results.

In chapter~\ref{open problems} we give some open problems.

\subsection{Notations} 

If we do not imply otherwise, we always use the constants $k$ and $N$ and base set $A \cup B$ for semiintersecting families. For easier notation we have also introduced $S_A=S \cap A$ and $S_B=S \cap B$ for any set $S \subset A \cup B$ in Definition~\ref{S_A}.

We call $1$-dimensional affine subspaces $\textit{affine lines}$ as defined in the statement of Proposition~\ref{prop}. We partition the affine lines of a $2$-dimensional vector space into $\textit{parallel classes}$ (two of the lines are in the same parallel class if their corresponding linear subspaces are the same, see Definition~\ref{affine} and the paragraphs after it).

We always use Calligraphic letters to refer to families of sets or parallel classes of affine lines. We use capital letters to denote sets and lowercase letters for elements in sets.  

$\mathsf{MOLS}(n)$ is the maximum number of mutually orthogonal Latin squares, defined in Definition~\ref{MOLSdef}. 

$\pi(x)$ denotes the number of primes not greater than $x$.

$\sim$ denotes the asymptotic equivalence of functions. That is, if given two functions, both defined on the positive reals or on the positive integers, then we say $f(x) \sim g(x)$ if and only if
$$\lim_{x \to \infty} \frac{f(x)}{g(x)}=1.$$
We will also use the little-o notation. For two functions, $f(x)$ and $g(x)$ defined on the positive reals, we say that $f(x)=o(g(x))$ if for every $\varepsilon>0$ there exists a $K$, with 
$$|f(x)|<\varepsilon g(x)$$
for all $x>K$. With $o(g(x))$ we denote the set of all functions with the above property. We say that $o(g(x))=o(h(x))$ if the set of functions corresponding to $o(g(x))$ and $o(h(x))$ are the same. 
\subsection{Preliminaries}

In our constructions in chapter~\ref{lower} we will use certain definitions from algebra, combinatorics and number theory. We will give a brief introduction to them along with the applied theorems in this subsection.

We will use four well-known mathematical definitions and related theorems: finite fields, finite projective geometries, mutually orthogonal Latin squares and affine subspaces.

We will not state the definition of a field, as we consider it well-known. A finite field is a field with a finite number of elements. We will use the following famous theorem.

\begin{Th} \label{finite fields}
There exists a finite field with $n$ elements, exactly if $n$ is a prime power.
\end{Th}

\cite{Lid} contains a detailed introduction to finite fields.

We also consider finite projective geometries well-known. A good introduction to the topic can be read here \cite{Alb}. We will require the following known theorem.

\begin{Th} \label{finite projective}
If $n$ is a prime power then there exists a finite projective plane of order $n$.
\end{Th}

\begin{Def}[Mutually orthogonal Latin squares] \label{MOLSdef}
A Latin square is an $n \times n$ square filled with $n$ different symbols, each occurring exactly once in every row and column. Two Latin squares are said to be orthogonal if, considering the two filling as an ordered pair in every entry of the table, every ordered pair occurs exactly once in the square. A family of Latin squares are mutually orthogonal if every pair of them are orthogonal.

We use the usual $\mathsf{MOLS}(n)$ notation for the maximum number of mutually orthogonal Latin squares of size $n \times n$.
\end{Def}

 Mutually orthogonal Latin squares are a widely studied topic in combinatorics. In \cite{Col} the authors wrote a brief survey about the best known constructions.  We will also need a related theorem.

\begin{Th} \label{projective}
There exist $n-1$ mutually orthogonal Latin squares of size $n$ if and only if there exists a finite projective plane of order $n$. 
\end{Th}

Finally, we define affine subspaces in a vector space.

\begin{Def}[Affine subspace] \label{affine}
Let $V$ be a vector space. Any subset of the form $x+U$, where $x$ is a vector in $V$ and $U$ is a linear subsapce of $V$, is referred to as an \it{affine subspace} of $V$.
\end{Def}

We will refer to the set of affine subspaces corresponding to a certain linear subspace as a parallel class. The elements of such a class are pairwise disjoint.

We will use these definitions only when the vector space is two dimensional. In this case we can state the following useful property.
\begin{Prop} \label{prop}
In a two dimensional vector space $V$, every two affine subspaces corresponding to different 1-dimensional linear subspaces (which we will call affine lines) have exactly one vector as their intersection. 
\end{Prop}

\begin{proof}
Let two affine lines be $U_1=L_1+p_1$ and $U_2=L_2+p_2$ where $L_1$ and $L_2$ are  two linear subspaces. First, let us observe that any two non-zero vectors from the two different linear subspaces will be independent.  Let us choose two non-zero vectors $x\in L_1, y\in L_2$. As these are independent, we can express any vector as the linear combination of these two vectors, thus $U_1=L_1+a_1x+b_1y=L_1+b_1y$ and $U_2=L_2+a_2x+b_2y=L_2+a_2x$. It follows that if we translate the original linear subspaces $L_1$ and $L_2$ with $a_2x+b_1y$ then we will get the affine lines $U_1$ and $U_2$ respectively, as $x\in L_1$ and $y\in L_2$. The linear subspaces had only the origin as their intersection, therefore when they are translated they will have only $a_2x+b_1y$ as their intersection, thus we have proved our claim.
\end{proof}

At the end of Chapter~\ref{lower} we will use the famous prime number theorem.

\begin{Th}[Prime number theorem] \label{pnt}
Denote by $\pi(x)$ the number of primes not greater than $x$. (We can consider this function defined on the positive reals.) Then
$$\pi(k)=\frac{k}{\log(k)}+o\left(\frac{k}{\log(k)}\right).$$
\end{Th} 

\subsection{Other graph products} \label{kneser}
Our main problem can be stated as the calculation of the clique number in the Xor product of Kneser graphs. In this section we examine the clique number in other graph products.
First we define the three most commonly used graph products. In all three definitions $G_1$ and $G_2$ are simple graphs.
\begin{Def}[Cartesian product]
Let $G_1\square G_2$ be the graph on the vertex set $V_1\times V_2$, where $V_1\times V_2$ is the Cartesian product of the vertex sets of the original graphs, and let $(g,h)$ and $(g',h')$ be connected, if either $g=g'$ and $h$ and $h'$ are connected in $G_2$, or $g$ and $g'$ are connected in $G_1$ and $h=h'$. 
\end{Def}
\begin{Def}[Tensor product]
Let $G_1\times G_2$ be the graph on the vertex set $V_1\times V_2$, and let $(g,h)$ and $(g',h')$ be connected, if both $g$ and $g'$ are connected in $G_1$ and $h$ and $h'$ are connected in $G_2$. Note that no vertex is connected with itself in $G_1$ and $G_2$. 
\end{Def}
\begin{Def}[Strong product]
Let $G_1\boxtimes G_2$ be the graph on the vertex set $V_1\times V_2$, and let $(g,h)$ and $(g',h')$ be connected, if either $g=g'$ and $h$ and $h'$ are connected in $G_2$, or $g$ and $g'$ are connected in $G_1$ and $h=h'$, or both $g$ and $g'$ are connected in $G_1$ and $h$ and $h'$ are connected in $G_2$. This is the union of the previous two graph products.
\end{Def}

In the case of these graph products the calculation of the clique number is trivial not only in the case of Kneser graphs, but in the general case as well.
The following well known equivalences describe the clique number in the canonical graph products.
\begin{Prop}
$\omega(G\square G)=\omega(G)$
\end{Prop}
\begin{proof}
Take a clique from $G\square G$ with $\omega(G)$ vertices. Without the loss of generality we can suppose that $(v,u_1)$ and $(v,u_2)$ are part of this maximal clique. The only points connected to both of these are the ones having $v$ as their first coordinate. Therefore all the elements in the clique have $v$ as their first coordinate, so in the other coordinate they must form a clique, thus $\omega(G\square G)\leq\omega(G)$.

Let us take a maximal clique $A$ from $G$. Then in $G\square G$ the following set forms a clique of size $\omega(G)$. Let us take all the vertices of the form $(v,a)$ for a fixed $v$, where $a\in A$. Therefore $\omega(G\square G)=\omega(G)$

\end{proof}

\begin{Prop}
$\omega(G\times G)=\omega(G)$
\end{Prop}
\begin{proof}
There are no two vertices connected of the form $(v,u_1)$ and $(v,u_2)$, as $v$ is not connected to itself. Therefore, in a clique every $v\in G$ can only appear once in the first coordinate. All of the $v\in G$ that appear in such a way should be connected, therefore $\omega(G\times G) \leq \omega(G)$.

This is sharp, as if we take a clique $A$ of size $\omega(G)$, and take the points $(v,v)$ where $v\in A$, we will get a suitably large clique.
\end{proof}

\begin{Prop}
$\omega(G\boxtimes G)=\omega(G)^2$
\end{Prop}
\begin{proof}
If we observe only the first coordinate in a clique in $G\boxtimes G$ we get that for each pair of vertices the first coordinate is either the same, or is connected. This means that the size of the set of possible first coordinates is at most $\omega(G)$. This is true for the second coordinate as well, therefore there are only $\omega(G)^2$ possibilities for an element in the clique.

We can construct a large enough clique by taking a maximal clique $A$ in $G$, and taking the clique consisting of the vertices $(u,v)$ where $u,v\in A.$
\end{proof}

\section{Upper bounds} \label{upper}

\noindent As stated in the introduction we start this section with $2$ simple lemmas, which give surprisingly strong bounds in some cases.

\begin{Lemma} \label{l1}
$$f(k,N)\leq 2k\left(\left\lfloor \frac{N}{k} \right\rfloor-1 \right)+1$$
\end{Lemma}

\begin{proof}
Let $\mathcal{S}$ be a semiintersecting family, and let $S \in \mathcal{S}$. We want to give an upper bound for the number of sets $T$ with $T \in \mathcal{S}$, $T \neq S$ and $S_A \cap T_A \neq \emptyset$. Let us denote the family of sets with these properties $\mathcal{T}$.

Note that for all $a \in A$ there are at most $\left\lfloor \frac{N}{k} \right\rfloor$ sets $H \in \mathcal{S}$ with $a \in H$, because any two such $H$ intersect in $A$, so they must be pairwise disjoint in $B$. Using this, for all $u \in S_A$ there are at most $\left\lfloor \frac{N}{k} \right\rfloor-1$ sets $T \in \mathcal{S}$ with $T \neq S$ and $u \in T$. $|S_A|=k$, and $u \in T$ for some $u\in S_A$ for all $T \in \mathcal{T}$, so 
$$|\mathcal{T|} \leq k\left(\left\lfloor \frac{N}{k} \right\rfloor-1 \right).$$
By the symmetry between $A$ and $B$ this is also an upper bound for the sets that intersect $S$ at $B$, and all sets from $\mathcal{S}$ must intersect $S$ either in $A$ or in $B$. The conclusion follows.
\end{proof}

\begin{Lemma} \label{l2}
$$f(k,N) \leq \left\lfloor \frac{N \left\lfloor \frac{N}{k} \right\rfloor}{k} \right\rfloor$$
\end{Lemma}

\begin{proof}
We will use double counting.

Let $\m{S}$ be a semiintersecting family. Using the same fact as in Lemma~\ref{l1}, for all $a \in A$ there can be at most $\lef \frac{N}{k} \rrr$ elements of $\m{S}$ that contain $a$. Therefore we can count the (ordered) pairs $(a,S)$, where $a \in A$, $S \in \m{S}$ and $a \in S$. On one hand for every $a$ there are at most $\lef \frac{N}{k} \rrr$ sets containing it, so there are at most $N\lef \frac{N}{k}\rrr$ pairs. On the other hand the number of pairs equals $|\m{S}|\cdot k$, since every $S \in \m{S}$ contains exactly $k$ elements in $A$.

Therefore we get$$|\m{S}| \leq \left\lfloor \frac{N \left\lfloor \frac{N}{k} \right\rfloor}{k} \right\rfloor.$$
\end{proof}

Now we turn to the case when $N$ is sufficiently large. We examine the $k=2$ case separately, because in this case we can calculate the exact value. Our next theorem is a sharp upper bound, and we will prove that this is indeed sharp in a slightly more general form in the next chapter.

In the next two theorems we will use the concept of Ramsey numbers. Ramsey's theorem \cite{ramsey1930problem} states that for all positive integers $a$ and $b$ there exist a smallest integer $R(a,b)$ (called Ramsey number) with the property that any simple graph with $R(a,b)$ vertices contains either a clique of size $a$ or an independent set of size $b$. From now on we use the notation $R(a,b)$ for Ramsey numbers.  

\begin{Th} \label{f(2,N)}
If $\lef \frac{N}{2} \rrr \geq R(7,7)$ then
$$f(2,N) \leq \left\lfloor \frac{N}{2} \right\rfloor + 4.$$
\end{Th}

\begin{proof}

Let $\m{S}$ be a semiintesecting family with maximal cardinality. 
Then from the trivial lower bound $|\m{S}| \geq \lef \frac{N}{2} \rrr \geq R(7,7)$. 

Consider the graph $G$ that has the sets in $\m{S}$ as its vertices and two vertices $S,T \in \m{S}$ are connected if and only if $(S_A) \cap (T_A) \neq \emptyset$.
From Ramsey's Theorem we know that this graph contains $7$ vertices that form a click or an independent set. 

In the former case, there are $7$ sets in $\m{S}$ pairwise intersecting in $A$. 
In the latter case, there are $7$ sets in $\m{S}$ pairwise disjoint in $A$, therefore pairwise intersecting in $B$. By the symmetry between $A$ and $B$ we can assume that there are $7$ sets pairwise intersecting in $A$. Let $\m{T} \subset \m{S}$ denote such a $7$-element subfamily. 

Let us call a $2$-element subset $P\subset A$ \textit{catcher}, if there exist at least $3$ sets $S_1,S_2,S_3 \in \m{S}$ such that $(S_1)_A=(S_2)_A=(S_3)_A=P$. Notice that if $T_A \cap P=\emptyset$ for some $T \in \m{S}$, then $T$ must intersect $S_1,S_2$ and $S_3$ in $B$ which is impossible, as they are disjoint in $B$ and $|T_B|=2$. So all $T \in \m{S}$ must intersect $P$. This is a pretty strong condition, so our goal is to find catcher subsets of $A$.

We have two possibilities for $\m{T}$. Either there is an element $c\in A$ with $c \in T$ for all $T \in \m{T}$ (\textbf{Case 1} and \textbf{Case 2}). If there is no such $c$, then there are subsets $T_A$ ($T \in \m{T}$) of the form $\{c,d\}$, $\{c,e\}$ and one that is not containing $c$, which must be $\{d,e\}$. All other sets in $\m{T}$ must intersect $\{c,d\},\{d,e\}$ and $\{e,c\}$, which is only possible, if $T_A=\{c,d\},\{d,e\}$ or $\{e,c\}$ for all $T\in \m{T}$ (\textbf{Case 3}).

\textbf{Case 1}: There is an element $c \in A$ that is contained in all $T \in \m{T}$ and $c \in S$ for all $S \in \m{S}$. All sets of $\m{S}$ intersect at $c$, so they are pairwise disjoint in $B$, which means that $|\m{S}| \leq \lef \frac{N}{2} \rrr$, and we are finished.

\textbf{Case 2}: There is an element $c \in A$ that is contained in all $T \in \m{T}$ but  $c \notin S$ for some $S \in \m{S}$. 
Let us choose such an $S \in \m{S}$ and call $d$ and $e$ the elements of $S_A$, i.e. $S_A=\{d,e\}$. The sets in $\m{T}$ are disjoint in $B$, and $S$ can intersect at most two of them in $B$ as $|S_B|=2$. Therefore $S_A$ must intersect all but at most two sets in $\m{T}$, which means that at least $5$ sets from $\m{T}$ contain $d$ or $e$. Consequently either $d$ or $e$ is contained in at least $3$ sets from $\m{T}$. We can assume that the former holds and conclude that $\{c,d\}\subset A$ is a catcher subset. 

\textbf{Case 3}: $T_A=\{c,d\}$, $\{c,e\}$ or $\{d,e\}$ for all $T\in \m{T}$. This means that one of them must occur at least $3$ times. We can assume that $\{c,d\}$ occurs at least $3$ times, and conclude that $\{c,d\}\subset A$ is a catcher subset.

So in both \textbf{Case 2} and \textbf{Case 3} we concluded that $P:=\{c,d\}$ is a catcher subset. Let $\m{C} \subset \m{S}$ with $C \in \m{C}$ exactly if $c \in C$ and $d \notin C$. If $|\m{C}| \leq 4$ then all but at most $4$ sets from $\m{S}$ contain $d$, and these sets are pairwise disjoint in $B$, so in this case $|\m{S}| \leq \lef \frac{N}{2} \rrr +4$, therefore the statement is true.

Thus we may assume that $|\m{C}|>4$. Similarly, let $\m{D} \subset \m{S}$ with $D \in \m{D}$ exactly if $c \notin D$ and $d \in D$. In the same way we can assume, that $|\m{D}|>4$. $\m{D}$ is nonempty, therefore we may choose $D\in \m{D}$. Call $e$ the other element of $D_A$, i.e. $D_A=\{d,e\}$.

All the sets of $\m{C}$ are disjoint in $B$, so $D$ can intersect at most two of them in $B$. Consequently it intersects at least $3$ of them, $C_1,C_2,C_3$, in $A$. Notice, that it is only possible if $(C_1)_A=(C_2)_A=(C_3)_A=\{c,e\}$, which means that $Q:=\{c,e\}$ is a catcher subset.

$S_A\cap P\neq \emptyset$ and $S_A\cap Q\neq \emptyset$ for all $S\in \m{S}$, so either $c \in S_A$ or $S_A=\{d,e\}$. If there are less than $3$ sets with $S_A=\{d,e\}$ then all but these contain $c$, so they are disjoint in $B$, therefore $|\m{S}|<\lef \frac{N}{2} \rrr+3$, and we are finished.

The other option is that $R:=\{d,e\}$ is also a catcher subset. 
Then $S_A\cap P\neq \emptyset$, $S_A\cap Q\neq \emptyset$ and $S_A\cap R\neq \emptyset$ for all $S \in \m{S}$. It is only possible if $S_A= P,Q$ or $R$ for all $S \in \m{S}$. So all set in $\m{S}$ pairwise intersect in $A$, thus they are pairwise disjoint in $B$, so $|\m{S}| \leq \lef \frac{N}{2} \rrr$. The proof is complete.

\end{proof}

And here comes the most important and challenging theorem of our article. We solve the problem for large $N$ up to a constant deviation. 

\begin{Th}\label{thm}
For all $k$ there exists a constant $c$ (depending only on $k$) with $f(k,N)\leq\left\lfloor \frac{N}{k} \right\rfloor+c$ for all $N$.
\end{Th}

\begin{Rem}
In a forthcoming work with Zoltán Füredi, among
others, we will show that $c= \Theta ( 4^k \sqrt{k})$.
\end{Rem}

\begin{proof}
Let $\m{S}$ be a semiintersecting family for some $k$, $N$. First we prove two lemmas.

\begin{Lemma}\label{2k+1}
Suppose that we have a set $C\subset A$ with $1\leq|C|\leq k-1$. Furthermore, let $S_1,S_2,\ldots ,S_{2k+1} \in \m{S}$ be sets such that the intersection of any two of them is exactly $C$. Then all of the other sets of $\m{S}$ must intersect some element of $C$.
\end{Lemma}

\begin{proof}
Since all of the sets $S_1,S_2,\ldots, S_{2k+1}$ intersect in $A$, they are disjoint in $B$. Suppose that for some set $S \in \m{S}$, $S \cap C= \emptyset$. We have that $S_i \cap S_j=C$ for all $i,j$, therefore $S_i \setminus C$ and $S_j \setminus C$ are disjoint. There are $2k+1$ sets that $S$ has to intersect either in $A$ or in $B$. Recall that $(S_i)_B$ and $(S_j)_B$ are disjoint for all $i,j$, and so are all the sets $(S_i)_A \setminus C $ and $(S_j)_A \setminus C$. This means that $S$ cannot intersect more than $k$ of them neither in $A$ nor in $B$, since it has only $k$ elements both in $A$ and in $B$, so we have a contradiction.

\end{proof}

We define two recursive sequences that will play an important role in Lemma~\ref{intersection lemma}.

\begin{Def}\label{seqences}
Define 
$$r_1:=k+1$$
and let 
$$r_i:=R(m_{i-1},2k+1)$$ 
for $i \geq 2.$

Furthermore, for all $i\geq 1$ define
$$m_i=r_i\cdot i.$$ 
Note that this is a well-defined recursive definition, as the two sequences determine each other. $r_1$ is given, $r_i$ determines $m_i$ and $m_i$ determines $r_{i+1}$.
\end{Def}

\begin{Lemma} \label{intersection lemma}
There exists an $m \in \mathbb{N}$ such that it depends only on $k$, and for any $\m{Z} \subset \m{S}$ with $|\m{Z}|\geq m$ there exists a $Z \in \m{Z}$ such that one of the following holds. 
\begin{itemize}
    \item $Z_A \cap S_A \neq \emptyset$ for all $S \in \m{S}$. 
    \item $Z_B \cap S_B  \neq \emptyset$ for all $S \in \m{S}$.
\end{itemize}
We will show that $m=R(m_k,m_k)$ is an appropriate choice.
\end{Lemma}

\begin{proof}

Throughout the proof we will use the sequences $r_i$ and $m_i$ defined in Definition~\ref{seqences}. 


Suppose, that $\m{Z} \subset \m{S}$ has at least $m=R(m_k,m_k)$ elements.

 We will construct the same graph as we did in the beginning of the proof of Theorem~\ref{f(2,N)}. 
 The vertices are the sets in $\m{Z}$, and two of them are connected if and only if they intersect in $A$. Since we supposed that $|\m{Z}|\geq R(m_k,m_k)$, our graph must either contain a click or an independent set of size $m_k$ because of the Ramsey theorem. This means that there must be a subset $\m{M}_{k} \subset \m{Z}$ such that $|\m{M}_{k}|\geq m_k$, and either $X_A \cap Y_A$ is nonempty for every $X,Y \in \m{M}_{k}$  or $X_B \cap Y_B $ is nonempty for every $X,Y \in \m{M}_k$. We can assume that the former holds, the other case is the same as the roles of $A$ and $B$ are symmetric.
 Now let us choose a set $X\in \m{M}_{k}$. Since  $Y_A\cap X_A$ is nonempty for every $Y\in \m{M}_{k}$, and since $|\m{M}_{k}|\geq m_k= k\cdot r_k$ by definition, there must be an element $z_1 \in X_A \subset A$, such that at least $r_k$ sets in $\m{M}_k$ contain $z_1$. Let us call the set of these sets $\m{R}_k$. Formally:
$$\m{R}_k:= \{M\in \m{M}_k:z_1\in M_A\}$$
We want to create a descending chain of subfamilies 
$$\m{Z} \supset \m{M}_k \supset \m{R}_k \supset \m{M}_{k-1} \supset \m{R}_{k-1} \supset \m{M}_{k-2} \supset \ldots \supset \m{R}_1$$
such that the following conditions hold.
 \begin{itemize}
     \item $|\m{M}_l| \geq m_l$ and $|\m{R}_l| \geq r_l$ for all $1 \leq l \leq k.$
    \item  For all $1 \leq l \leq k$ there exist $z_1, z_2, \ldots, z_{k-l} \in A$ such that $\{z_1, z_2, \ldots, z_{k-l} \} \subset M$ for all $M \in \m{M}_l$ and $\{z_1, z_2, \ldots, z_{k-l} \} \subsetneq M \cap N$ for all $M,N \in \m{M}_l$.
    \item For all $1 \leq l \leq k$ there exist $z_1, z_2, \ldots, z_{k-l+1} \in A$ with $\{z_1, z_2, \ldots, z_{k-l+1} \} \subset R$ for all $R \in \m{R}_l$.
 \end{itemize}

We do this by induction. For the base step we already created the families $\m{M}_k$ and $\m{R}_k$, it is easy to check that these satisfy the conditions of the induction. For the induction step we assume that we already created $\m{M}_l$ and $\m{R}_l$ ($2 \leq l \leq k$), and then we try to construct the subfamilies $\m{M}_{l-1}$ and $\m{R}_{l-1}$.

Thus we have $\m{M}_l$ and $\m{R}_{l}$ such that $\{z_1,z_2,\ldots, z_{k+1-l}\}\subset X$ for every $X \in \m{R}_{l}$ and $|\m{R}_{l}|\geq r_{l}$.

We create a graph again: the vertices are the sets from $\m{R}_{l}$, and two of them, $X$ and $Y$, are connected if and only if $\{z_1,z_2,\ldots, z_{k+1-l}\} \subsetneq X_A \cap Y_A$. Since $|\m{R}_l|\geq r_{l}=R(m_{l-1},2k+1) $, we have either a click of size $m_{l-1}$ or an independent set of size $2k+1$.

In the latter case, we can use Lemma~\ref{2k+1}, as we have an independent set of size $2k+1$, which means that the corresponding sets pairwise intersect in exactly $\{z_1,z_2,\ldots,z_{k+1-l}\}$. Therefore, we conclude that for every set $X \in \m{S}$, $X_A\cap \{z_1,z_2,\ldots, z_{k+1-l}\}\neq \emptyset$, in particular any $Z \in \m{R}_l$ proves the lemma as $\{z_1,z_2, \ldots, z_{k+1-l}\} \subset Z$.

In the former case, let us define $\m{M}_{l-1}$ to be a subfamily of $\m{R}_{l}$ such that for every pair $X,Y \in \m{M}_{l-1}$ it holds that $\{z_1,z_2, \ldots, z_{k+1-l}\} \subsetneq X_A \cap Y_A$, and among these subfamilies it has maximal cardinality. We have $|\m{M}_{l-1}| \geq m_{l-1}$, since we had a click of size $m_{l-1}$ in the graph we defined.

We created $\m{M}_{l-1}$ and it indeed satisfies the conditions stated above. Now we want to find a subfamily of $\m{M}_{l-1}$ that is appropriate for $\m{R}_{l-1}$. We will do the same thing we did in the base case when we created $\m{R}_k$ from $\m{M}_k$.  

Let us choose a set $X \in \m{M}_{l-1}$. All $M \in \m{M}_{l-1}$ must intersect $X_A\setminus \{z_1,z_2, \ldots, z_{k+1-l}\}$, and $|X_A\setminus \{z_1,z_2, \ldots, z_{k+1-l}\}|=l-1$. By definition $m_{l-1}=r_{l-1}\cdot (l-1)$, so we can find an element ${z_{k+1-(l-1)}}$ in  $X_A$ (with $z_{k+1-(l-1)} \neq z_1,z_2, \ldots, z_{k+1-l}$) such that at least $r_{l-1}$ sets from $\m{M}_{l-1}$ contain it. Here we used that $k+1-l<k$, and therefore we could choose some element other than $z_1,z_2, \ldots, z_{k+1-l}$. Now define
$$\m{R}_{l-1}:=\{M\in \m{M}_{l-1}: z_{k+1-(l-1)}\in M_A\}.$$

We got a family of sets, $\m{R}_{l-1}$, such that $\{z_1,z_2, \ldots, z_{k+1-(l-1)}\} \subset X$ for every set $X \in \m{R}_{l-1}$ and $|\m{R}_{l-1}| \geq r_{l-1}$. These were the conditions for $\m{R}_{l-1}$, therefore the induction works.

Throughout the induction, if at any step the graph we created has an independent set of size $2k+1$ then we have proved the lemma.

If this case never occurs, so we always find a large clique, then we terminate in the case when $l=1$. Therefore, we get a family $\m{R}_1$ and elements $z_1, z_2, \ldots, z_k \in A$ with $|\m{R}_1|\geq r_1=k+1$ and $\{z_1,z_2, \ldots, z_k\} \subset X_A$ for all $X\in \m{R}_1$. Meaning that for all $X \in \m{R}_1$ we have $X_A=\{z_1,z_2, \ldots, z_k\}$. Now all $S \in \m{S}$ must intersect $\{z_1,z_2, \ldots, z_k\}$, because otherwise it intersects all the sets from $\m{R}_{1}$ in $B$, which is clearly impossible as these are at least $k+1$ sets which must be pairwise disjoint in $B$, and $|S_B|=k$. For any $Z \in \m{R}_1$ we have $\{z_1,z_2, \ldots, z_k\} \subset Z$, therefore it is a good choice for the lemma. So by choosing $m=R(m_k,m_k)$ the lemma is indeed true.

\end{proof}

Now we turn to prove Theorem~\ref{thm}.
We will prove that $|\m{S}|\leq \lef \frac{N}{k}\rrr +m-1$. 

Suppose on the contrary that $|\m{S}|\geq m+\lef\frac{N}{k}\rrr$.
Then $|\m{S}|\geq m$, so we can apply Lemma~\ref{intersection lemma} with $\m{Z}=\m{S}$, and conclude that there exists a set $Z \in \m{Z}$, such that either $Z_A\cap S_A \neq \emptyset$ for all $S\in \m{S}$ or $Z_B\cap S_B \neq \emptyset$ for all $S\in \m{S}$. We can suppose without loss of generality that the former holds, and define the following nonempty family of sets.
$$\m{F}=\{X\in \m{S}: X_A\cap S_A\neq \emptyset \text{ for all } S \in \m{S} \}$$

We know, that $F_A\cap G_A$ is nonempty for every $F,G\in \m{F}$, therefore they are disjoint in $B$, so $|\m{F}|\leq \lef\frac{N}{k}\rrr$.

Consequently, $|\m{S}\setminus\m{F}|\geq m$, so we can apply Lemma~\ref{intersection lemma} again for $\m{Z}=\m{S}\setminus \m{F}$. The roles of $A$ and $B$ are not symmetric anymore, therefore we have to check both cases of the lemma.

\textbf{Case 1}: There exists a set $X\in \m{S}\setminus\m{F}$ such that $S_A\cap X_A\neq\emptyset$ for all $S \in \m{S}$. However, this was the definition of the family $\m{F}$, and $X$ is not in $\m{F}$, so this is a contradiction.

\textbf{Case 2}: There exists a set $X\in \m{S}\setminus\m{F}$ such that $S_B\cap X_B\neq\emptyset$ for all $S \in \m{S}$.
Let us choose a set $F$ from $\m{F}$. $F \neq X$ as $F \in \m{F}$ and $X \notin \m{F}$. Notice that $F_A\cap X_A\neq \emptyset$ because of the definition of $\m{F}$, and $F_B\cap X_B \neq \emptyset$ because of the definition of $X$. This contradicts $\m{S}$ being a semiintersecting family, so we proved Theorem~\ref{thm}.

\end{proof}

\section{Lower bounds} \label{lower}

\noindent We start with three simple observations. 

The first is the trivial lower bound $\lef \frac{N}{k} \rrr$. Just take a family $\m{S}$ with $S_A=T_A$ and $(S_B) \cap (T_B)=\emptyset$ for all $S,T \in \m{S}$.  This is actually not as weak as it looks, see Theorem~\ref{thm}.

Secondly, if $k$ is fix and $N_1<N_2$ then $f(k,N_1) \leq f(k,N_2)$, as a semiintersecting family for $(k,N_1)$ can be considered as a semiintersecting family for $(k,N_2)$ where we do not use the remaining elements. So in other words, $f(k,N)$ is monotonously increasing in $N$. 

The third observation is a bit more complex, so we state it as a lemma.

\begin{Lemma} \label{osztodas}
Assume that we have a semiintersecting family $\m{S}$ with constants $(k,N)$. Let $g: A \cup B \to \ZZ^+$ with the following properties.

\noindent There exists a constant $N'$ with
$$\sum_{a \in A} g(a)=\sum_{b \in B} g(b)=N',$$
and there exists a constant $k'$ with 
$$\sum_{s \in S_A} g(s) = \sum_{s \in S_B} g(s)=k'$$
for all $S \in \m{S}$.

\noindent Then $f(k',N') \geq |\m{S}|$.
\end{Lemma}

\begin{proof}
If we understand the statement, the proof is evident. Intuitively this means that we replace every point $p$ with $g(p)$ other points. 

For all $c \in A \cup B$ assign a set $H_c$ such that $|H_c|=g(c)$ and $H_i \cap H_j=\emptyset$ for all distinct $i, j \in A \cup B$. Let $A'=\bigcup_{a \in A} H_a$ and $B'=\bigcup_{b \in B} H_b$. For all $S \in \m{S}$ consider the set $S'=\bigcup_{s \in S} H_s$, and let $\m{S'}$ be the family of these sets for all $S \in \m{S}$. 

From the assumptions of the lemma we know that $|A'|=|B'|=N'$ and $|S'_{A'}|=|S'_{B'}|=k'$ for all $S' \in \m{S'}$, and clearly $\m{S'}$ is semiintesecting because $\m{S}$ is semiintersecting, therefore
$$f(k',N') \geq |\m{S'}|=|\m{S}|.$$
\end{proof}

\begin{Cor}
If $d$ divides both $k$ and $N$ then $f(k,N) \geq f \left(\frac{k}{d},\frac{N}{d} \right)$.
\end{Cor}

\begin{proof}
Let us use Lemma~\ref{osztodas} for a semiintersecting family $\m{S}$ with constants $\left(\frac{k}{d},\frac{N}{d} \right)$ and with maximal cardinality among such, i.e. $|\m{S}|=f \left(\frac{k}{d},\frac{N}{d} \right)$. Furthermore let $g \equiv d$ be the constant function. With these conditions Lemma~\ref{osztodas} gives the desired result.
\end{proof}

Now we will investigate the special case when $N=k^2$, furthermore we only consider constructions with an extra condition beside the conditions of semiintersecting families.

\begin{Def} \label{Latin}
Call a semiinterseting family $\m{S}$ Latin, if $N=k^2$ and we can partition $B$ into $k$ disjoint sets $B_1$, $B_2$,\ldots, $B_k$, all of them having exactly $k$ elements with the property that for all $S \in \m{S}$ we have $S_B=B_i$ for some $1 \leq i \leq k$, and for all $1 \leq i \leq k$ there is either no $S \in \m{S}$ with $S_B=B_i$ or there are exactly $k$ such $S \in \m{S}$.
\end{Def}

At first glance this definition could look quite arbitrary, but the motivation is that the Latin semiintersecting families can be corresponded to mutually orthogonal Latin squares. The next theorem formulates this more precisely.

\begin{Th} 
Let $k \geq 2$. The maximal number that can be the cardinality of a Latin semiintersecting family with parameters $(k,k^2)$ is $k \cdot \min(\mathsf{MOLS}(k)+2,k)$.
\end{Th}

\begin{proof}
Take a Latin semiintersecting family $\m{S}$ with maximal cardinality. Let us call $B_1, B_2, \ldots , B_k$ the disjoint sets in Definition~\ref{Latin}. Index them such a way that exactly the first $l$ among them are the sets with $k$ different $S\in\m{S}$ intersecting $B$ at that set, i.e. $S_B = B_i$ for $k$ sets $S \in \m{S}$ if $1 \leq i \leq l$, and there is no such $S \in \m{S}$ if $l+1 \leq i \leq k$. For all $1 \leq i \leq l$ call $\m{S}_i \subset \m{S}$ the family of sets that intersect $B$ at $B_i$. Note that for all $1 \leq i \leq l$ the sets $S \in \m{S}_i$ are pairwise disjoint in $A$, and there are $k$ of them, so they must partition $A$. Also if $i \neq j$ and $X \in \m{S}_i$, $Y \in \m{S}_j$ then $X_A \cap Y_A \neq \emptyset$. This is true for all $Y \in \m{S}_j$, so $X$ must intersect every set $Y \in \m{S}_j$ in exactly $1$ element. So every pair of sets from $\m{S}$ that belong to different $\m{S}_i$ intersect in $1$ element.

Clearly $l \geq 2$. Arrange the points of $A$ in a $k \times k$ square such that the rows are the subsets in the partition obtained from $\m{S}_1$, and the columns are the subsets in the partition obtained from $\m{S}_2$. It is clearly possible to construct such an arrangement. For example let us call the sets from $\m{S}_1$ and $\m{S}_2$ respectively $X_1,X_2,\ldots, X_k$ and $Y_1,Y_2,\ldots,Y_k$. If $(X_i)_A\cap (Y_j)_A=a\in A$ then put $a$ in the $i$th row and $j$th column. Now for every $3 \leq i \leq l$ we can see from the above conditions that every $X_A \in \m{S}_i$ has to intersect every row and column exactly once, so the partition obtained from $\m{S}_i$ is a Latin square. Also the above observations for $3 \leq i<j \leq l$ tells us exactly that $\m{S}_i$ and $\m{S}_j$ are orthogonal Latin squares. 

Conversely if $m\leq k-2$ mutually orthogonal Latin squares are given, it is possible to construct a Latin semiintersecting family of size $m+2$ from them. The construction can be obtained by doing everything in the opposite direction. Arrange $A$ into a $k\times k$ square. Let the rows and the columns of $A$ be the sets from $\m{S}_1$ and $\m{S}_2$ respectively, and assign to every Latin square a semiintersecting subfamily $\m{S}_i$. We needed the condition that $m\leq k-2$, since $B$ is partitioned into $k$ parts, thus we can only create $k$ subfamilies $\m{S}_i\subset \m{S}$ satisfying the above conditions even if $\mathsf{MOLS}(k)>k-2$. The theorem follows.

\end{proof}

In particular the following corollary is true.

\begin{Cor} \label{MOLS}
$f(k,k^2) \geq k \cdot \min(\mathsf{MOLS}(k)+2,k)$
\end{Cor}

With Corollary~\ref{MOLS} we can give lower bounds for $f(k,k^2)$. As Latin semiintersecting families are a very special class of semiintersecting families this estimation looks quite weak at first glance, but actually we can see from Lemma~\ref{l2} that it is sharp if $\mathsf{MOLS}(k) \geq k-2$. In particular, by Theorem~\ref{projective} it is sharp if there exists a projective plane of order $k$, consequently from Theorem~\ref{finite projective} it is sharp for prime powers. We will state this result as a separate lemma and write down another proof, as it is nice and simple, and we will use the same ideas and notations in later lemmas and theorems.

\begin{figure}[h]
    \centering 
\includegraphics[trim=0.7cm 1.0cm 0cm 1cm, clip,scale=0.5]{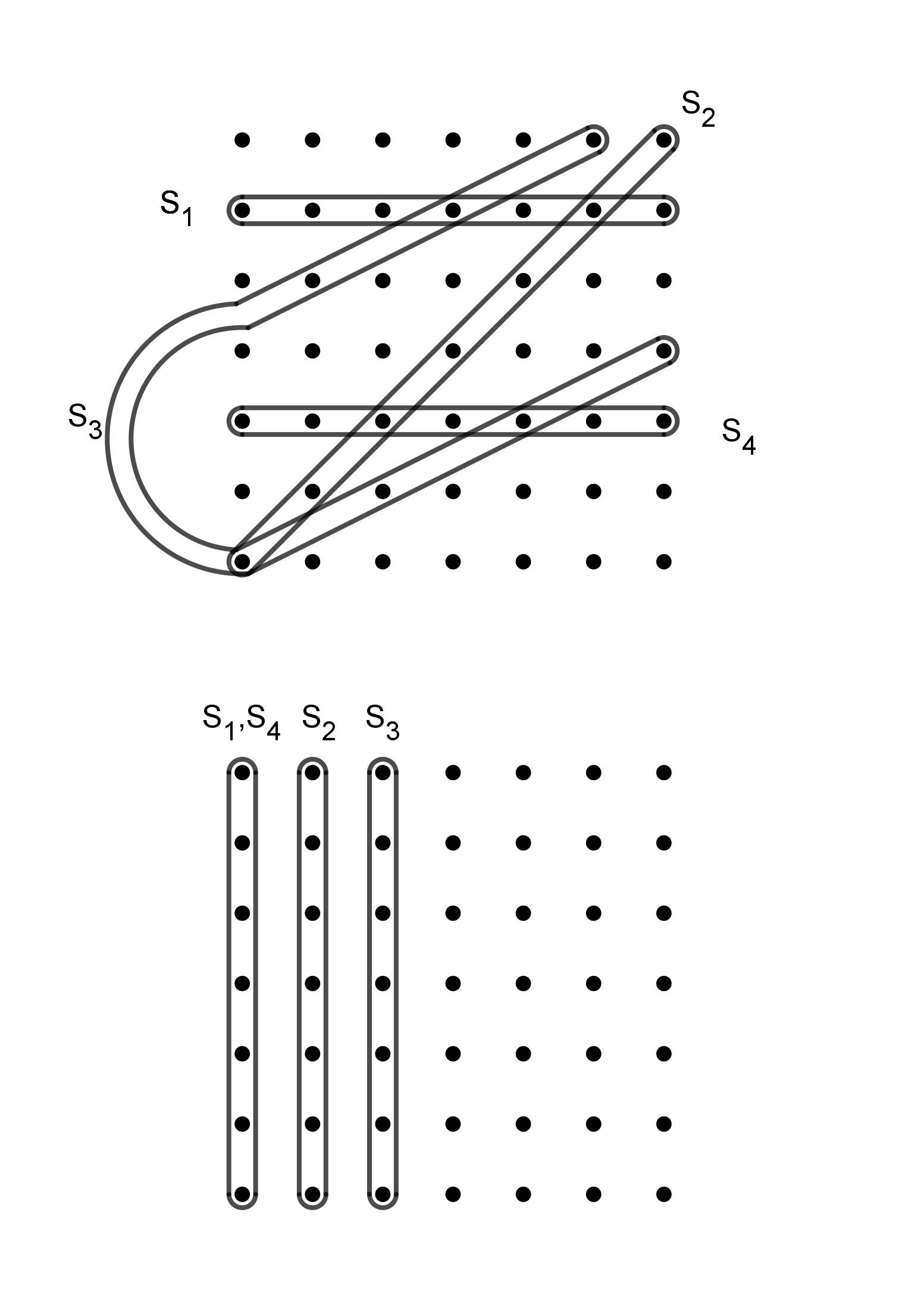}
\caption{This figure shows the construction detailed in Lemma~\ref{p} for $p=7$. The set $A$ corresponds to the set of points of $\F_p^2$. $(S_i)_A$ shows examples of affine lines in this space. In the set $B$ every parallel class of lines in $A$ will have a corresponding subset. As $(S_1)_A$ and $(S_4)_A$ are from the same parallel class, their corresponding subset is the same in $B$, or in other words $(S_1)_B=(S_4)_B$.} 
    \label{fig1}
\end{figure}

\begin{Lemma} \label{p}
Assume that $p$ is a prime power. Then
$$f(p,p^2) \geq p^2.$$
\end{Lemma}

\begin{proof}

Let $\F_p$ be the $p$-element field (that exists by Theorem~\ref{finite fields}), and let $D$ be a $2$-dimensional vector space over this field. Let us take the affine lines (remember that affine lines are affine subspaces corresponding to  $1$-dimensional linear subspaces) in this space. There are $p+1$ parallel classes of affine lines corresponding to the $p+1$ distinct $1$-dimensional linear subspaces, let us call these parallel classes $\m{D}_1$, $\m{D}_2$,\ldots, $\m{D}_{p+1}$. All of them contain $p$ affine lines.  From Proposition~\ref{prop} any pair of affine lines from different classes have exactly one point as their intersection.

Consider the following construction for a semiintersecting family $\m{S}$ with constants $k=p$ and $N=p^2$. Assign the elements of $A$ to the points of $D$, and divide $B$ into $p$ disjoint sets (call them $B_1$, $B_2$,\ldots, $B_p$) all of them containing $p$ elements.

For all $1 \leq i \leq p$ define
$$\m{S}_i:=\{B_i\cup X: X\in \m{D}_i\},$$
and define $\m{S}$ as $\bigcup_{i=1}^p\m{S}_i$. This is $p^2$ sets as $|\m{D}_i|=p$ for all $1 \leq i \leq p$. By the construction it is trivial that $|S_A|=|S_B|=p$ for all $S \in \m{S}$. The other condition of semiintersecting families also holds, because we have two cases:
\begin{itemize}
    \item $X,Y\in \m{S}_i$. In this case, $X_A\cap Y_A=\emptyset$, as they correspond to parallel affine lines and $X_B\cap Y_B=B_i$.
    \item $X\in\m{S}_i$ and $Y\in \m{S}_j$ ($i\neq j$). In this case, $X_A\cap Y_A\neq \emptyset$ as they correspond to nonparallel affine lines and $X_B\cap Y_B=\emptyset$.
\end{itemize}

 So this is indeed a semiintersecting family with $p^2$ elements which proves our claim.
\end{proof}

Now we give $3$ generalizations of this lemma. The first is interesting in itself, but we will also use it as a lemma for the theorem after it.

\begin{Lemma} \label{l}
Assume that $p$ is a prime power and $l\leq p+1$. Then
$$f(p,lp^2) \geq lp^2.$$
\end{Lemma}

\begin{proof}
We use similar notations as in the proof of Lemma~\ref{p}, as we are going to generalize the construction from that lemma. So now, as the statement of the lemma implies, we construct a semiintersecting family $\m{S}$ with constants $k=p$ and $N=lp^2$ with $|\m{S}|=lp^2$. 

Divide $A$ into $l$ disjoint sets with $p^2$ elements, and assign the $2$-dimensional vector spaces (over $\F_p$) $D^1$, $D^2$,\ldots, $D^l$ to them. Let $\m{D}^i_1$, $\m{D}^i_2$,\ldots, $\m{D}^i_{p+1}$ be the parallel classes of affine lines in $D^i$ for all $1 \leq i \leq l$. In $B$ we will only use $p^2$ elements, so let $B' \subset B$ with $|B'|=p^2$. Assign the points of a $2$-dimensional vector space $E$ (over $\F_p)$ to the elements in $B'$ and let $\m{E}_1$, $\m{E}_2$,\ldots, $\m{E}_{p+1}$ be the parallel classes of affine lines in $E$. For all $1 \leq i \leq l$ let us take $\m{S}$ from Lemma~\ref{p} with $D=D^i$ and with $B_1$, $B_2$,\ldots, $B_p$ as the sets from $\m{E}_i$, call this family $\m{S}_i$. Let
$$\m{S}:=\bigcup_{i=1}^l \m{S}_i.$$
Obviously $|\m{S}|=lp^2$.

We need to check that this is a semiintersecting family. 
It is trivial that for all $S \in \m{S}$ we have $|S_A|=|S_B|=p$. Now we check the second condition of semiintersecting families. 
Take two sets $X$ and $Y$ from $\m{S}$.
\begin{itemize}
    \item If $X,Y\in \m{S}_i$, then by Lemma~\ref{p} the condition holds.
    \item If $X\in \m{S}_i$ and $Y\in \m{S}_j$, then they are disjoint in $A$ and they intersect in $B$, 
as their intersections with $B$ are two nonparallel affine lines in $E$.
\end{itemize}
We have proved that $\m{S}$ is semiintersecting. 
\end{proof}

\begin{Remark}
This lemma is not sharp (or at least we cannot prove it is), however combining it with Lemma~\ref{l1} we have
$$lp^2 \leq f(p,lp^2) \leq 2p(lp-1)+1,$$
so the upper bound is less than twice the lower bound. 
\end{Remark}

The next theorem gives a lower bound when $N$ is large if $p$ is a fixed prime power. It is sharp for $p=2$, but unfortunately it is very likely that it is far from the best for larger $p$ prime powers. Still, it is much better than the trivial $\lef \frac{N}{p} \rrr$ estimation.

\begin{Th} \label{bigNprimek}
Let $p$ be a prime power. If $N \geq p^3$ then
$$f(p,N) \geq \left\lfloor \frac{N}{p} \right\rfloor-p^2+p^3.$$
\end{Th}

\begin{proof}
We use the notations from Theorem~\ref{l}. Let $p$ be a prime power and $N \geq p^3$, we construct an appropriate semiintersecting family $\m{S}$ with constants $(p,N)$. Let $A' \subset A$ and $B' \subset B$ be subsets with $|A'|=|B'|=p^3$, and take the family $\m{S}$ from Theorem~\ref{l} with $l=p$ and $A=A'$, $B=B'$. In this proof we call this family $\m{S}_1$.

Notice that in Theorem~\ref{l} we had $p+1$ parallel classes of affine lines in $E$. Until now we only used $p$ from them as $l=p$ in this case.

Let $E \in \m{E}_{p+1}$. Take $\lef \frac{N-p^3}{p} \rrr$ pairwise disjoint $p$-element subsets from $A \setminus A'$, let $\m{C}$ be the family containing these subsets. Let
$$\m{S}_2:=\{C \cup E : C \in \m{C}\}.$$

Now let $\m{S}=\m{S}_1 \cup \m{S}_2$. Then
$$|\m{S}|=|\m{S}_1|+|\m{S}_2|=p^3+\lef \frac{N-p^3}{p} \rrr=\lef \frac{N}{p} \rrr-p^2+p^3,$$
we prove that $\m{S}$ is semiintersecting. Trivially $|S_A|=|S_B|=p$ for all $S \in \m{S}$. Now we prove the second condition of semiintersecting families. Take any $2$ sets $X$ and $Y$ from $\m{S}$.
\begin{itemize}
    \item If $X,Y\in \m{S}_1$, then the condition is true by Theorem~\ref{l}.
    \item If $X\in \m{S}_1$ and $Y \in \m{S}_2$, then they are disjoint in $A$ and they intersect in $B$, since $Y_B=E$ and $X_B$ is also an affine line not from the family $\m{E}_{p+1}$ (to construct $\m{S}_1$ we only used $\m{E}_1$, $\m{E}_2$,\ldots, $\m{E}_{p}$).
    \item If $X,Y\in \m{S}_2$, then they obviously intersect in $B$ and are disjoint in $A$.
\end{itemize}
 The proof is finished. 
\end{proof}

\begin{Remark} \label{clower}
In other words in this theorem we proved that $c \geq p^3-p^2$ for any prime power $p$ where $c$ is the constant in the statement of Theorem~\ref{thm} with $k=p$. 
\end{Remark}

The last generalization of Lemma~\ref{p} gives a sharp bound (combined with Lemma~\ref{l2}) for a larger class of pairs $(k,N)$.

\begin{Th} \label{pk}
For all positive integers $k$ and prime powers $p$ with $p \leq k$ the following inequality holds.
$$f(k,pk) \geq p^2$$
\end{Th}

\begin{proof}
We will use Lemma~\ref{osztodas} on the construction obtained in Lemma~\ref{p}. So again, take the construction from Lemma~\ref{p}, call it $\m{S}$, and use the notations introduced in that lemma. Let $X \in \m{D}_{p+1}$ be an affine line. Note that in the construction we only used the first $p$ parallel classes, so $S_A$ is an affine line intersecting $X$ for every $S \in \m{S}$. Consequently, from Proposition~\ref{prop} we have $|S_A \cap X|=1$ for all $S \in \m{S}$.  Take an element $b_i \in B_i$ for $1 \leq i \leq p$. Define $g$ such that 
$$g(c)=\left\{
                \begin{array}{ll}
                  k-p+1 & \text{if $c \in X$ or $c=b_i$ for some $1 \leq i \leq p$,} \\
                  1 & \text{otherwise.} 
                \end{array}
              \right.$$

Now we have to check that the conditions of Lemma~\ref{osztodas} are fulfilled. 
Let 
$$C=\{c \in A \cup B \ : \ \text{$c \in X$ or $c=b_i$ for some $1 \leq i \leq p$\}}.$$
In other words $C \subset A \cup B$ consist of the elements $c$ for which $g(c)=k-p+1$. 
$|C \cap A|=p$ as $|X|=p$ and $|C \cap B|=p$ as $C \cap B=\{b_1, b_2, \ldots, b_n$\}. Consequently,
$$\sum_{a \in A} g(a)=\sum_{b \in B} g(b)=(p^2-p)+p \cdot (k-p+1)=pk.$$

For any $S \in \m{S}$ we saw that $|S_A \cap C|=|S_A \cap X|=1$, and also $|S_B \cap C|=1$, as $S_B=B_i$ for some $1 \leq i \leq p$, so $S_B \cap C=\{b_i\}$. Consequently,
$$\sum_{s \in S_A} g(s) = \sum_{s \in S_B} g(s)= (p-1)+(k-p+1)=k$$
for all $S \in \m{S}$. Therefore, we can indeed apply Lemma~\ref{osztodas} which finishes the proof.

\end{proof}

\begin{Cor}\label{k^2}
$f(k,k^2) \sim k^2$
\end{Cor}
\begin{proof}
Combining Theorem~\ref{pk} with the monotonicity of $f(k,N)$ on $N$, we know that for every $k$ and prime power $p \leq k$ the following inequality holds.
$$f(k,k^2) \geq f(k,pk)=p^2.$$
From Lemma~\ref{l2} we also know that $f(k,k^2)\leq k^2$.

Now we will use the same argument as in \cite[p. 494]{hardy75} about the density of prime numbers.
Let $0 < \varepsilon <1$. As $\log(1-\varepsilon)$ is a constant, evidently $$\log(k) \sim \log(k)+\log(1-\varepsilon),$$ thus it is easy to see that
$$o\left(\frac{k}{\log(k)}\right)=o\left(\frac{(1-\varepsilon)k}{\log(k)+\log(1-\varepsilon)}\right)=o\left(\frac{(1-\varepsilon)k}{\log((1-\varepsilon)k)}\right).$$
Using this equation and the prime number theorem (Theorem~\ref{pnt})
$$\pi(k)-\pi((1-\varepsilon)k)=\frac{k}{\log(k)}-\frac{k-\varepsilon k}{\log(k)+\log(1-\varepsilon)}+o\left(\frac{k}{\log(k)}\right).$$
Notice that from $\log(k) \sim \log(k)+\log(1-\varepsilon)$ we have
$$\frac{k}{\log(k)}-\frac{k}{\log(k)+\log(1-\varepsilon)}=\frac{k}{\log(k)} \cdot \left( 1-\frac{\log(k)}{\log(k)+\log(1-\varepsilon)} \right)=o\left(\frac{k}{\log(k)}\right).$$ 
Therefore
$$\pi(k)-\pi((1-\varepsilon)k)=\frac{\varepsilon k}{\log(k)+\log(1-\varepsilon)}+o\left(\frac{k}{\log(k)}\right).$$

There exists a $K_1$ such that $\log(k)>\log(1-\varepsilon)$ for $k>K_1$. There exists a $K_2$ such that for all $k > K_2$ the following inequality holds by the definition of the little-o notation.
$$\pi(k)-\pi((1-\varepsilon)k) > \frac{\varepsilon k}{\log(k)+\log(1-\varepsilon)}-\frac{\varepsilon}{2} \cdot \frac{k}{\log(k)}$$
Now for all $k>\max(K_1, K_2)$ we have
$$\pi(k)-\pi((1-\varepsilon)k) > \frac{\varepsilon k}{2\log(k)}-\frac{\varepsilon}{2} \cdot \frac{k}{\log(k)}=0.$$

We proved that for any $0 < \varepsilon < 1$ there exists a sufficiently large $K$ with the property that if $k>K$ then $\pi(k)-\pi((1-\varepsilon)k) > 0$. This means that there exists a prime number $k(1-\varepsilon)\leq p\leq k$. Therefore $$k^2 \geq f(k,k^2)\geq p^2 \geq k^2(1-\varepsilon)^2$$ for any given $0<\varepsilon<1$ and sufficiently large $k$, thus $$1 \geq \lim_{k\to\infty} \frac{f(k,k^2)}{k^2} \geq (1-\varepsilon)^2.$$ This holds for all $\varepsilon$, therefore $$1 = \lim_{k\to\infty}\frac{f(k,k^2)}{k^2} .$$
\end{proof}

\section{Summary} \label{summary}

\noindent Combining the results from Section~\ref{upper} and Section~\ref{lower} we know quite much about the order of magnitude of $f(k,N)$.

The most important result in our opinion is that in Theorem~\ref{thm} for fix $k$ we determined that the exact value of $f(k,N)$ is $\lef \frac{N}{k} \rrr$ up to a constant deviation depending only on $k$. In particular, for a fix $k$ we asymptotically determined $f(k,N)$.

If we consider only the pairs $(k,k^2)$ then Corollary~\ref{k^2} shows that $f(k,k^2) \sim k^2$. 

Also, from Theorem~\ref{pk} we can see that for a fixed prime power $p$ we have $f(k,pk) \sim p^2$ (in fact we know that they are equal if $k\geq p$).

We found the exact result for the following $(k,N)$ pairs.

\begin{itemize}
    \item If $k=2$ and $\lef \frac{N}{2}\rrr \geq R(7,7)$ then the upper and lower bounds from Theorem~\ref{f(2,N)} and Theorem~\ref{bigNprimek} meet, so we have
    $$f(2,N) = \left\lfloor \frac{N}{2} \right\rfloor + 4.$$
    \item If $k$ is a positive integer and $p$ is a prime power not bigger than $k$ then from Lemma~\ref{l2} and Theorem~\ref{pk} 
    $$f(k,pk) = p^2.$$
\end{itemize}

\section{Open problems} \label{open problems}

\noindent In this section we summarize the remaining unanswered questions. 

The next natural step in the study of our problem will be the following generalization.
\begin{?}
What is the clique number in higher powers of Kneser graphs with respect to the Xor product?
\end{?}  
\noindent This question is notably harder than the case we examined in this article.

We have seen that Lemma~\ref{l2} is sharp in a variety of cases if $N \leq k^2$. Therefore, the following question arises naturally.

\begin{Question}
\label{N<k^2}
Is the following equation true for all $N \leq k^2$?
$$f(k,N) = \left\lfloor \frac{N \left\lfloor \frac{N}{k} \right\rfloor}{k} \right\rfloor$$
\end{Question} 

Checking the equality for small values is rather hard due to the problem's computational complexity with simple methods. 
Though we do not have much computational data for Question~\ref{N<k^2}, but by the examination of the structure of semiintersecting families we think that it might be true.



Another remaining aspect of our problem is the behaviour of $f$ when $N>k^2$. We know the trivial lower bound $\lef \frac{N}{k} \rrr$, and from Lemma~\ref{l1} we have a rather strong upper bound.  Consider the special case when $N=dk^n$. The results in our article determine the asymptotics of this case when $n=2$ and $d=1$ and when $n=1$ and $d=p$ for prime power $p$. However, when $n\geq 3$ we have little to no knowledge about the asymptotics of~$f$.

\begin{?}
What is the asymptotics of $f(k,dk^n)$?
\end{?}

\section*{Acknowledgement}
\noindent The authors would like to thank Gyula O.H. Katona for both his help in providing the problem and his useful comments.

\bibliography{biblio}
\bibliographystyle{plain}

\bigskip

\noindent
{\bf András Imolay}\\
E\"otv\"os L\'or\'and University \\
H-1117 Budapest, Pázmány Péter sétány 1/C \\
\texttt{imolay.andras[at]gmail.com}
\medskip
\ \\
{\bf Anett Kocsis}\\
E\"otv\"os L\'or\'and University \\
H-1117 Budapest, Pázmány Péter sétány 1/C \\
\texttt{sakkboszi[at]gmail.com}
\medskip
\ \\
{\bf Ádám Schweitzer}\\
E\"otv\"os L\'or\'and University \\
H-1117 Budapest, Pázmány Péter sétány 1/C \\
\texttt{adamschweitzer1[at]gmail.com}

\end{document}